\documentclass[11pt]{amsart}
\usepackage{latexsym,amsmath,amssymb,amsfonts,amscd,graphics,appendix,amsxtra}
\usepackage[mathscr]{eucal}
\usepackage{fullpage}
\usepackage{xr-hyper}
\usepackage{tikz-cd}

\usepackage{hyperref}

\newtheorem{theorem}{Theorem}[section]
\newtheorem{proposition}[theorem]{Proposition}
\newtheorem{lemma}[theorem]{Lemma}

\newtheorem{corollary}[theorem]{Corollary}

\newtheorem{D}[theorem]{Definition}
\newenvironment{definition}{\begin{D} \rm }{\end{D}}
\newtheorem{R}[theorem]{Remark}
\newenvironment{remark}{\begin{R}\rm }{\end{R}}
\newtheorem{E}[theorem]{Example}

\numberwithin{equation}{section}

\newcommand{\Ar}{\mathbb{R}}
\newcommand{\Cee}{\mathbb{C}}

\newcommand{\Ker}{\operatorname{Ker}}

\newcommand{\Hom}{\operatorname{Hom}}
\newcommand{\Ext}{\operatorname{Ext}}

\newcommand{\im}{\operatorname{Im}}

\newcommand{\Spec}{\operatorname{Spec}}

\newcommand{\scrO}{\mathcal{O}}

\newcommand{\hY}{\widehat{Y}}
 \newcommand{\ccdot}{{\scriptscriptstyle{\bullet}}}
\newcommand{\uOb}{\underline\Omega^\ccdot}
\newcommand{\uOp}{\underline\Omega^p}

\title{Unobstructed deformations for singular Calabi-Yau varieties}

\begin{document}
\author[R. Friedman]{Robert Friedman}
\email{rf@math.columbia.edu}
\address{Columbia University, Department of Mathematics, New York, NY 10027}


\begin{abstract} Let $Y$ be  a compact Gorenstein analytic space with only isolated singularities and trivial dualizing sheaf. A recent paper of Imagi studies the deformation theory of $Y$ in case the singularities of $Y$ are weighted homogeneous and rational and $Y$ is K\"ahler. In this note, assuming that $H^1(Y;\scrO_Y) =0$, we generalize Imagi's results to the case where the singularities of $Y$ are Du Bois, with no assumption that they be weighted homogeneous, and where the K\"ahler assumption is replaced by the hypothesis that there is a resolution of singularities of $Y$ satisfying the $\partial\bar\partial$-lemma. As a consequence, if the singularities of $Y$ are additionally local complete intersections, then the deformations of $Y$ are unobstructed. The   log Calabi-Yau and Fano cases are also discussed.
\end{abstract}

\bibliographystyle{amsalpha}
\maketitle

\section{Introduction}

Let $M$ be a compact complex manifold. The nearby deformations of the complex structure on $M$ are described by Kuranishi's theorem: there exists a power series $\Psi$, the \textsl{Kuranishi map},  defined in a neighborhood of the origin in $H^1(M;T_M)$ and with values in $H^2(M;T_M)$, such that $\Psi(0) =0$ and the deformation space of complex structures on $M$, viewed as a germ, is isomorphic to the germ $(\Psi^{-1}(0),0)$. If $\Psi=0$, the deformations of $M$ are said to be \textsl{unobstructed}. The Bogomolov-Tian-Todorov theorem says that, if $M$ is a Calabi-Yau manifold, then  the deformations of $M$ are unobstructed. A very short proof, using only the property that $K_M\cong \scrO_M$ and the degeneration of the Hodge-de Rham spectral sequence, can be given by applying the $T^1$ lifting property of   Ran \cite{Ran-1}, Kawamata \cite{Kawamata}, and Deligne (unpublished).

There are analogues of $H^1(M;T_M)$ and  $H^2(M;T_M)$ in case the complex manifold is replaced by a  compact analytic space $Y$, namely the vector spaces $\mathbb{T}^1_Y$ and $\mathbb{T}^2_Y$. There is again a Kuranishi map, in this case a power series $\Psi$ defined in a neighborhood of the origin in $\mathbb{T}^1_Y$ and with values in $\mathbb{T}^2_Y$, such that $\Psi(0) =0$ and the deformation space of   $Y$, viewed as a germ, is isomorphic to the germ $(\Psi^{-1}(0),0)$. If $\Psi=0$, the deformations of $Y$ are again said to be \textsl{unobstructed}. If $Y$ is singular,  there can be local obstructions to deforming the singularities as well as global ones. If however  the singularities of $Y$ are local complete intersection (lci) singularities, then there are no local obstructions to deforming them. 

By analogy with the case of Calabi-Yau manifolds, we make the following definition:

\begin{definition}\label{defineCY}  A \textsl{singular Calabi-Yau variety} $Y$ is  a (connected) compact Gorenstein singular analytic space with trivial dualizing sheaf $\omega_Y$, i.e.\ $\omega_Y \cong \scrO_Y$.
\end{definition}

If $Y$ is a singular Calabi-Yau variety, it is often important to know if the deformations of $Y$ are unobstructed. One  motivation is to construct smoothings of $Y$. Note however that the property that the deformations of $Y$ are unobstructed does not necessarily imply that $Y$ is smoothable (and conversely, the existence of a smoothing  does not necessarily imply that the deformations of $Y$ are unobstructed). To prove smoothing results given that the deformations of $Y$ are unobstructed, one typically has to know more about the image of $\mathbb{T}^1_Y$ in $H^0(Y; T^1_Y)$, where $T^1_Y$ is a sheaf on $Y$ supported on the singular locus which measures the first order change in deformations around the singularities. Some results along these lines may be found in \cite{FL} and \cite{Tenie}.

 A  basic result in the deformation theory of singular Calabi-Yau varieties  is the following result of  Kawamata \cite{Kawamata}, Ran \cite{Ran}, and Tian \cite{Tian}:

\begin{theorem}\label{thmK} Let $Y$ be a singular Calabi-Yau variety with only ordinary double points such that $\omega_Y\cong \scrO_Y$ and there exists a resolution $\hY \to Y$ which is K\"ahler. Then the deformations of $Y$ are unobstructed. 
\end{theorem}

The proofs  are based on the $T^1$ lifting property as well as (for the proof in \cite{Ran}) the  related $T^2$ injective property of \cite{Ran0}. Here, the hypothesis that there is a K\"ahler resolution of $Y$ could be replaced by: there exists a resolution of $Y$ satisfying the $\partial\bar\partial$-lemma.  Note that Ran stated his result for somewhat more general singularities. However, there are several points in the papers \cite{Ran} and \cite{Tian} which need clarification, and the method of \cite{Tian} seems to work only in dimension $3$.

In dimension $3$, there is the following result of  Namikawa \cite{namtop},  which allows more complicated singularities:

\begin{theorem}\label{thmN} Let $Y$ be a singular Calabi-Yau variety of dimension $3$,  all of whose singular points are  isolated rational  lci singularities,    such that $H^1(Y;\scrO_Y) = 0$, and such that there exists a resolution $\hY \to Y$ which is K\"ahler. Then the deformations of $Y$ are unobstructed.
\end{theorem} 

As with Theorem~\ref{thmK}, we can relax the condition that there exist a K\"ahler resolution to the condition that there exists a resolution  satisfying the $\partial\bar\partial$-lemma. In \cite[Theorem 3.7]{FL25a}, this result has been partially generalized to the case of Du Bois singularities. There are also related results of  Gross \cite{gross_defcy}, where $Y$ has rational but not necessarily lci singularities, which connect  the deformation theory  of $Y$ to the local obstructions to deformations of the isolated singular points.

In higher dimensions, the author and Radu Laza have proved the following generalization of Theorem~\ref{thmK} \cite[Corollary 1.5]{FL22c}:

\begin{theorem}\label{unob11} Let $Y$ be a singular Calabi-Yau variety of dimension $n$ such that all singularities of $Y$ are $1$-Du Bois lci singularities and such that there exists a K\"ahler resolution of singularities of $Y$.   Then the deformations of $Y$ are unobstructed.
\end{theorem}

Here, we refer to \cite[Definition 3.6]{FL22c} and the references in that paper for the definition of $1$-Du Bois. In particular, lci $1$-Du Bois singularities are automatically rational, by \cite{FL22d} in the case of isolated lci singularities and by a result of Chen-Dirks-Musta\c{t}\u{a} for a general lci singularity \cite{ChenDirksM}. Note that the proof of Theorem~\ref{unob11} does not in fact require that the singularities be isolated as long as there is a K\"ahler resolution of singularities of $Y$. As with Theorem~\ref{thmN},  if the singularities are isolated, the K\"ahler assumption   can be replaced by the assumption that there exists a resolution  satisfying the $\partial\bar\partial$-lemma. The key point in the proof is that, if $Y$ has $1$-Du Bois lci singularities and under some hypotheses on the resolution, then for any  deformation   $\mathcal{Y}$ of  $Y$ over $\Spec A$, where $A$ is an Artin local $\Cee$-algebra, the $A$-modules $H^i(\mathcal{Y}; \Omega^1_{\mathcal{Y}/\Spec A})$ are free and commute with base change in an appropriate sense. This then implies the $T^1$ lifting property. 

Recently, Y.\ Imagi has posted a paper which proves, among other matters, the following \cite[Theorem 1.2]{Imagi}:

\begin{theorem}\label{Imagithm} Let $Y$ be a singular Calabi-Yau variety, all of whose singularities are isolated weighted homogeneous lci rational singularities, and such that $Y$ is K\"ahler in a strong sense. Then the deformations of $Y$ are unobstructed.
\end{theorem}

Imagi also deals with the case where the singularities are not necessarily lci along the lines of \cite{gross_defcy}. 

The hypotheses of Theorem~\ref{Imagithm} are not ideal. The assumption of weighted homogeneous singularities is very restrictive, and the K\"ahler condition seems hard to verify unless $Y$ is actually projective. The method of proof involves choosing model metrics in a neighborhood of the singular points and gluing them in to form a new metric on $Y$. The proof then proceeds via harmonic analysis using the new metric, and finally there is a discussion of deformation theory over Artin local $\Ar$-algebras, so as to extend the K\"ahler metric to deformations over such an $\Ar$-algebra. It seems hard to extend these methods to a more general setting. Nonetheless, there are two insights that can be gleaned from his approach:

\begin{enumerate}
\item  In proving $T^1$ lifting or   $T^2$ injectivity,  the $1$-Du Bois hypothesis is   too strong. It is not important that 
$H^i(\mathcal{Y}; \Omega^1_{\mathcal{Y}/\Spec A})$ is a locally free $A$-module and commutes with base change for \emph{all}  $i$. In the lci case, the main necessary ingredient   is  that the map
$H^{n-2}(\mathcal{Y}; \Omega^1_{\mathcal{Y}/\Spec A}) \to H^{n-2}(Y; \Omega^1_Y)$
is surjective. In the non lci case, one considers instead the map 
$$H^{n-2}(\mathcal{Y}; \Omega^1_{\mathcal{Y}/\Spec A}) \to H^{n-2}(Y; \Omega^1_Y)/\im  H^{n-2}_Z(Y; \Omega^1_Y),$$
where  $Z$ is the  singular locus of $Y$.  In dimension $3$, this idea already appears in \cite[Proof of Theorem 1]{namtop} in the lci case and in   \cite{gross_defcy} more generally. 
\item It is useful to look at compactly supported cohomology, which in this case is the relative cohomology $H^i(Y,Z)$, where  $Z$ is again the  singular locus of $Y$.  Given a good resolution $\pi\colon \hY \to Y$ with exceptional set $E =\pi^{-1}(Z)$, $H^i(Y,Z)$ is the same as the relative cohomology $H^i(\hY, E)$.  In this setting, $H^i(\hY, E)$ carries a mixed Hodge structure and we can use standard  methods to prove $E_1$ degeneration for an appropriate spectral sequence.
\end{enumerate} 

Armed with these extra ideas, we first work under the simplifying assumption of local complete intersection singularities for the sake of clarity, and prove the following in Section~\ref{section3}:

\begin{theorem}\label{mainthm} Let $Y$ be a singular Calabi-Yau variety of dimension $n \ge 3$, such that:
\begin{enumerate}
\item[\rm(i)]  The $\partial\bar\partial$-lemma holds for some resolution of singularities $\hY$ of $Y$.
\item[\rm(ii)]  The singularities of $Y$ are isolated, lci and Du Bois (i.e.\ $0$-Du Bois).
\item[\rm(iii)]  $H^1(Y; \scrO_Y) = 0$.
\end{enumerate} Then the deformations of $Y$ are unobstructed.
\end{theorem}

\begin{remark}   (i) Steenbrink proved in \cite[Proposition 3.7]{Steenbrink} that an isolated rational singularity is Du Bois, so Theorem~\ref{mainthm}  covers the case of rational singularities as a special case. Note that a normal Gorenstein singularity is rational $\iff$ it is canonical \cite[11.1]{Kollar}. By a theorem of Ishii \cite[proof of Theorem 2.3]{Ishii85}, an isolated Gorenstein singularity of dimension at least $2$ is Du Bois $\iff$ it is log canonical. In particular, in (ii) above we could replace ``Du Bois" with ``log canonical." 

\smallskip
\noindent (ii) The assumption $H^1(Y; \scrO_Y) = 0$ is typically satisfied if $Y$ has isolated singularities (but is not smooth). For the case where $Y$ is projective with  rational singularities, this follows from a result of Kawamata \cite[Theorem 8.3]{KawamataFiber}. In this case, an easy argument due to Namikawa (which holds in somewhat greater generality) shows that $\mathbb{T}^0_Y = H^0(Y; T^0_Y) =0$, i.e.\ the deformation functor is prorepresentable.  For completeness, we recall the proof  in  Section~\ref{section3} (Lemma~\ref{Namlemma}).

\smallskip
\noindent (iii) If $Y$ is a compact analytic space with isolated singularities and if there exists  a resolution $\pi\colon \hY \to Y$ satisfying the $\partial\bar\partial$-lemma, then, by \cite[Remark 5.2]{FL},  the same is true for all such resolutions.
\end{remark}

In \S\ref{subsection3.2}, we discuss how to use the results of Section~\ref{section2}  in order for them to cover    the non lci case as well. The  consequences for the deformation theory of $Y$ are recalled in Theorems~\ref{conseqgenT1} and \ref{genT1thm}.

Many of these ideas generalize to the log Calabi Yau case, i.e.\ the case of a log deformation of a pair $(Y,D)$, where $Y$ has isolated Gorenstein Du Bois singularities and $D$ is a simple normal crossing Cartier divisor contained in the smooth locus of $Y$ and such that $\omega_Y \cong \scrO_Y(-D)$, again under the special assumption that $H^1(Y; \scrO_Y) = 0$. In case $Y$ is smooth, there are various unobstructedness results due to, among others, \cite{Ran-1}, \cite{KKP}, \cite{Sano2}, \cite{Iacono}. We give a version of their results in the case of a singular $Y$ in \S\ref{subsection3.3}. For example, the natural analogue of Theorem~\ref{mainthm} holds for the case of a log Calabi-Yau pair (Theorem~\ref{pairmainthm}) with lci singularities. In the Fano case, i.e.\ if $\omega_Y^{-1}$ is ample, then an easy argument shows that, if $Y$ has $1$-Du Bois lci singularities (not necessarily isolated), then $\mathbb{T}^i_Y =0$ for $i \ge 2$ and in particular $ \mathbf{Def}_Y$ is unobstructed \cite[Theorem 4.5]{FL}.  This result has been extended by Tenie to the case where the singularities are isolated Du Bois lci singularities \cite{Tenie}. In \S\ref{subsection3.3},   we show the following (Theorem~\ref{Fanomainthm}): If $Y$ is    a compact   analytic space of dimension at least $3$ with isolated lci rational singularities  such  that $\omega_Y^{-1}$ is nef and big,  then $ \mathbf{Def}_Y$ is  unobstructed.

 \subsection*{Acknowledgements} It is a pleasure to thank Yohsuke Imagi for sharing his preprint \cite{Imagi} and Radu Laza and Johan de Jong for many conversations about this and other subjects. I would especially like to thank Mark Gross for extended correspondence related to the generalized $T^1$ lifting criterion.

\subsection*{Notation and conventions} We work over $\Cee$ with the classical topology. $Y$ will always denote  a compact, reduced and normal  Gorenstein analytic  variety of dimension $n \ge 3$,  with isolated singularities, not necessarily satisfying the Calabi-Yau condition,   $Z$ will denote the singular locus of $Y$ (a finite set), and  $\pi\colon \hY \to Y$  will denote a good resolution of $Y$ with exceptional set $E= \pi^{-1}(Z)$, a divisor with simple normal crossings.

\section{Results from Hodge theory} \label{section2} 

In this section, we  do not assume that  $Y$  necessarily satisfies the  Calabi-Yau condition.

 \subsection{Mixed Hodge structures on relative cohomology}\label{ssect2.1}   A general reference for this section is \cite[\S7.3]{PS}. Let   $\uOb_{Y,Z} = R\pi_*\Omega^\ccdot_{\hY}(\log E)(-E)$ be the relative filtered de Rham complex, with associated graded terms $\uOp_{Y,Z} = R\pi_*\Omega^p_{\hY}(\log E)(-E)$. In degree $0$, $\pi_*\scrO_{\hY}(-E) = I_Z$, the ideal of the subset $Z$. With this notation, the singularities of $Y$ are   \textsl{Du Bois}   if   $R^i\pi_*\scrO_{\hY}(-E) = 0$ for $i >0$.  If $i\colon Z \to Y$ is the inclusion, define $i_!\Cee_{Y-Z}$ as the kernel of the surjection $\Cee_Y \to i_*\Cee_Z$,  where $\Cee_Y$ and $\Cee_Z$ are the constant sheaves on $Y$ and $Z$ respectively. Set
$$\Omega^\ccdot_{Y,Z} = \Ker\{ \Omega^\ccdot_Y \to i_*\Cee_Z[0] \} = \{I_Z \to  \Omega^1_Y\to  \Omega^2_Y \to \cdots\}.$$
There  are natural morphisms of complexes 
$$i_!\Cee_{Y-Z} \to \Omega^\ccdot_{Y,Z} \to R\pi_*\Omega^\ccdot_{\hY}(\log E)(-E) = \uOb_{Y,Z}$$ such that the composition $i_!\Cee_{Y-Z} \to    \uOb_{Y,Z}$ is a quasi-isomorphism. 

Suppose that the $\partial\bar\partial$-lemma holds for $\hY$. The hypercohomology spectral sequence with $E_1$ term
$$E_1^{p,q} = H^q(\hY; \Omega^p_{\hY}(\log E)(-E)) \implies \mathbb{H}^{p+q}
(\hY; \Omega^\ccdot_{\hY}(\log E)(-E)) \cong H^{p+q}(\hY, E; \Cee) \cong H^{p+q}(Y, Z; \Cee)$$
degenerates at $E_1$. More precisely, let $f\colon  \Omega^\ccdot_{\hY} \to  j_*\Omega^\ccdot_E/\tau^\ccdot_E$ be the natural surjection, where $j\colon E \to \hY$ is the inclusion, $\Omega^\ccdot_E$ is the complex of K\"ahler differentials on $E$ and $\tau^\ccdot_E$ is the subcomplex of differentials supported on the singular locus of $E$. There is an exact sequence
$$0 \to \Omega^\ccdot_{\hY}(\log E)(-E) \to \Omega^\ccdot_{\hY}\to j_*\Omega^\ccdot_E/\tau^\ccdot_E \to 0$$
which gives rise to a filtered quasi-isomorphism 
$$\Omega^\ccdot_{\hY}(\log E)(-E)[1] \simeq_{\rm{qis}} \operatorname{Cone}^\ccdot f.$$
Then the $E_1$ degeneration and the fact that  $\mathbb{H}^k
(\hY; \Omega^\ccdot_{\hY}(\log E)(-E)) \cong H^k(\hY, E; \Cee)$ are implied by \cite[Theorem 3.22, Proposition 5.46, and Example 7.25]{PS}.

 \subsection{The case of a local complete intersection} Our goal is to prove the following: 
 
\begin{theorem}\label{surjthm}  Suppose   that
\begin{enumerate}
\item[\rm(i)] The $\partial\bar\partial$-lemma holds for $\hY$.
\item[\rm(ii)] The singularities of $Y$ are isolated, lci and Du Bois.
\item[\rm(iii)] $H^{n-1}(Y; \scrO_Y) = 0$.
\end{enumerate}
Let $A$ be an Artin local $\Cee$-algebra  and let $\mathcal{Y} \to \Spec A$ be a   deformation of $Y$ over $\Spec A$. Then for every quotient $A' = A/I$, with $\mathcal{Y}' = \mathcal{Y}\times_{\Spec A}A'$, the natural homomorphism $H^{n-2}(\mathcal{Y}; \Omega^1_{\mathcal{Y}/\Spec A})\to H^{n-2}(\mathcal{Y}'; \Omega^1_{\mathcal{Y}'/\Spec A'})$ is surjective.
\end{theorem}
\begin{proof} We begin with a series of lemmas:

\begin{lemma}\label{lemma0.11} The natural map $\Omega^1_Y  \to \pi_*\Omega^1_{\hY}(\log E)(-E)$ is an isomorphism.
\end{lemma}
\begin{proof} This follows from the fact that, since  the singularities of $Y$ are isolated lci of dimension at least $3$, the sheaf $\Omega^1_Y$ is reflexive. (See e.g.\ \cite[Proposition 9.7]{Kunz} as well as \cite[Remark 3.20]{FL22c}.)
\end{proof}

\begin{lemma}\label{lemma0.2} \begin{enumerate}
\item[\rm(i)]   $H^i(\hY; \scrO_{\hY}(-E)) \cong H^i(Y; \pi_* \scrO_{\hY}(-E)) \cong H^i(Y; I_Z)$, and the homomorphism  
$$H^i(\hY; \scrO_{\hY}(-E)) \cong H^i(Y; I_Z)\to H^i(Y; \scrO_Y) $$ is surjective for $i=1$ and an isomorphism for $i>1$.  In particular, $H^{n-1}(\hY; \scrO_{\hY}(-E)) =0$. 
\item[\rm(ii)]  The map $H^{n-2}(Y; \Omega^1_Y) \to H^{n-2}(\hY; \Omega^1_{\hY}(\log E)(-E))$ is injective.
\item[\rm(iii)]
  The differential $d\colon H^{n-2}(Y; \scrO_Y) \to H^{n-2}(Y; \Omega^1_Y)$ is zero.
  \end{enumerate}
\end{lemma}
\begin{proof} (i) This follows from the Leray spectral sequence, the Du Bois assumption,  and the fact that the map $H^i(Y; I_Z) \to H^i(Y; \scrO_Y)$ is surjective for $i=1$ and an isomorphism for $i > 1$.

\smallskip
\noindent (ii)  By a result of Steenbrink \cite[Theorem 5]{SteenbrinkDB}, if $n \ge 4$, the assumption of lci singularities implies that $R^{n-3}\pi_*\Omega^1_{\hY}(\log E)(-E) = 0$. From the Leray spectral sequence and Lemma~\ref{lemma0.11}, there is an exact sequence
$$H^0(Y; R^{n-3}\pi_*\Omega^1_{\hY}(\log E)(-E)) \to H^{n-2}(Y; \Omega^1_Y) \to H^{n-2}(\hY; \Omega^1_{\hY}(\log E)(-E)).$$
Thus the map $H^{n-2}(Y; \Omega^1_Y) \to H^{n-2}(\hY; \Omega^1_{\hY}(\log E)(-E))$ is injective. The same holds if $n=3$ because then the map $H^1(Y; \Omega^1_Y) = H^1(Y; \pi_*\Omega^1_{\hY}(\log E)(-E))\to H^1(\hY; \Omega^1_{\hY}(\log E)(-E))$ is automatically injective.

\smallskip
\noindent (iii)   There is a commutative diagram
$$\begin{CD}
H^{n-2}(\hY; \scrO_{\hY}(-E))@<{\cong}<<   H^i(Y; \pi_* \scrO_{\hY}(-E)) = H^{n-2}(Y;I_Z)  @>>>  H^{n-2}(Y; \scrO_Y) \\
@V{d}VV @V{d}VV @VV{d}V \\
H^{n-2}(\hY; \Omega^1_{\hY}(\log E)(-E)) @<<< H^{n-2}(Y; \pi_*\Omega^1_{\hY}(\log E)(-E) )@<{\cong}<<    H^{n-2}(Y; \Omega^1_Y).
\end{CD}$$
By (i), the top right horizontal map is an isomorphism if $n \ge 4$ and is surjective if $n=3$. The left hand vertical map is $0$. By (ii), the map $H^{n-2}(Y; \pi_*\Omega^1_{\hY}(\log E)(-E) ) \to H^{n-2}(Y; \Omega^1_{\hY}(\log E)(-E))$ is injective. Thus $d\colon H^{n-2}(Y; \scrO_Y) \to H^{n-2}(Y; \Omega^1_Y)$ is zero.
\end{proof} 

\begin{lemma}\label{lemma0.3}  Let $\mathcal{C}^\ccdot_Y$ be the complex $\{\scrO_Y \xrightarrow{d} \Omega^1_Y\}$ and let $\mathcal{C}^\ccdot_{Y,Z}$ be the complex $\{I_Z  \xrightarrow{d} \Omega^1_Y\}$. Then
\begin{enumerate}
\item[\rm(i)]   $\mathbb{H}^{n-1}(Y;\mathcal{C}^\ccdot_{Y,Z} ) \cong \mathbb{H}^{n-1}(Y; \mathcal{C}^\ccdot_Y)\cong H^{n-2}(Y; \Omega^1_Y)$. 
\item[\rm(ii)]    The homomorphism
$i_!\Cee_{Y-Z}  \to \mathcal{C}^\ccdot_{Y,Z}$ induces a surjection
$$H^{n-1}(Y,Z; \Cee) \to \mathbb{H}^{n-1}(Y;\mathcal{C}^\ccdot_{Y,Z}) \cong H^{n-2}(Y; \Omega^1_Y).$$
\item[\rm(iii)]  The homomorphism
$\Cee\to \mathcal{C}^\ccdot_Y$ induces a surjection
$$H^{n-1}(Y; \Cee) \to \mathbb{H}^{n-1}(Y;\mathcal{C}^\ccdot_Y) \cong H^{n-2}(Y; \Omega^1_Y).$$
\end{enumerate}
\end{lemma}
\begin{proof} (i)  The isomorphism  $\mathbb{H}^{n-1}(Y;\mathcal{C}^\ccdot_{Y,Z} ) \cong \mathbb{H}^{n-1}(Y;\mathcal{C}^\ccdot_Y)$ follows from the exact sequence
$$0 \to \mathcal{C}^\ccdot_{Y,Z} \to \mathcal{C}^\ccdot_Y \to i_*\Cee_Z[0] \to 0.$$
The   hypercohomology spectral sequence gives an exact sequence
$$H^{n-2}(Y; \scrO_Y) \xrightarrow{d}  H^{n-2}(Y; \Omega^1_Y) \to \mathbb{H}^{n-1}(Y;\mathcal{C}^\ccdot_Y) \to H^{n-1}(Y; \scrO_Y) .$$
By assumption, $H^{n-1}(Y; \scrO_Y) =0$. By Lemma~\ref{lemma0.2}(iii), $d\colon H^{n-2}(Y; \scrO_Y) \to H^{n-2}(Y; \Omega^1_Y)$ is zero. Thus $\mathbb{H}^{n-1}(Y;\mathcal{C}^\ccdot_Y) \cong H^{n-2}(Y; \Omega^1_Y)$.

\smallskip
\noindent (ii) The homomorphism $i_!\Cee_{Y-Z}  \to \mathcal{C}^\ccdot_{Y,Z} $   defines a homomorphism 
$$H^{n-1}(Y,Z; \Cee) \to \mathbb{H}^{n-1}(Y;\mathcal{C}^\ccdot_{Y,Z} )=H^{n-2}(Y; \Omega^1_Y).$$
To see that this homomorphism  is surjective, let $\mathcal{D}^\ccdot$ be the complex  $\{\scrO_{\hY}(-E)  \to \Omega^1_{\hY}(\log E)(-E)\}$. 
By  the $E_1$ degeneration described in \S\ref{ssect2.1}, the homomorphism
$$H^i(Y,Z; \Cee) = \mathbb{H}^i(\hY; \Omega^p_{\hY}(\log E)(-E)) \to \mathbb{H}^i(\hY; \Omega^p_{\hY}(\log E)(-E)/\sigma^{\ge 2}) = \mathbb{H}^i(\hY; \mathcal{D}^\ccdot)$$
is surjective. Moreover, the hypercohomology spectral sequence for $\mathbb{H}^i(\hY; \mathcal{D}^\ccdot)$ degenerates at $E_1$, so that in particular there is an exact sequence
$$0 \to H^{n-2}(\hY; \Omega^1_{\hY}(\log E)(-E)) \to \mathbb{H}^{n-1}(\hY;\mathcal{D}^\ccdot) \to H^{n-1}(\hY; \scrO_{\hY}(-E)) \to 0.$$
Hence, by Lemma~\ref{lemma0.2}(i), $\mathbb{H}^{n-1}(\hY;\mathcal{D}^\ccdot) \cong H^{n-2}(\hY; \Omega^1_{\hY}(\log E)(-E))$. There is a morphism of complexes $ \mathcal{C}^\ccdot_{Y,Z} \to R\pi_*\mathcal{D}^\ccdot $ and hence a commutative diagram
$$\begin{CD}
H^{n-1}(Y,Z;\Cee) @>>> \mathbb{H}^{n-1}(Y;\mathcal{C}^\ccdot_{Y,Z} ) \cong H^{n-2}(Y; \Omega^1_Y)\\
@| @VVV\\
H^{n-1}(Y,Z;\Cee) @>>> \mathbb{H}^{n-1}(\hY;\mathcal{D}^\ccdot) \cong H^{n-2}(\hY; \Omega^1_{\hY}(\log E)(-E)).
\end{CD}$$
The lower horizontal arrow is surjective. Also, by Lemma~\ref{lemma0.2}(ii), the right hand vertical map is injective. Thus $H^{n-1}(Y,Z; \Cee) \to  H^{n-2}(Y; \Omega^1_Y)$ is surjective.

\smallskip
\noindent (iii) There is a commutative diagram
$$\begin{CD} 
H^{n-1}(Y,Z;\Cee) @>>>  \mathbb{H}^{n-1}(Y;\mathcal{C}^\ccdot_{Y,Z} ) \cong H^{n-2}(Y; \Omega^1_Y)\\
@V{\cong}VV @|\\
H^{n-1}(Y; \Cee) @>>> \mathbb{H}^{n-1}(Y; \mathcal{C}^\ccdot_Y) \cong H^{n-2}(Y; \Omega^1_Y).
\end{CD}$$
Thus the result follows from (ii).
\end{proof} 

\begin{remark} The proof of (ii) above shows that the map 
$$H^{n-2}(Y; \Omega^1_Y) \cong H^{n-2}(Y; \pi_*\Omega^1_{\hY}(\log E)(-E)) \to H^{n-2}(\hY; \Omega^1_{\hY}(\log E)(-E))$$ is an isomorphism. See \cite[Proposition 5.2]{PSV} for a generalization of the argument. 
\end{remark}

\begin{lemma}\label{lemma0.4} Let $A$ be an Artin local $\Cee$-algebra and let $\mathcal{Y} \to \Spec A$ be a deformation of $Y$ over $\Spec A$. Then the natural map $H^{n-2}(\mathcal{Y}; \Omega^1_{\mathcal{Y}/\Spec A}) \to H^{n-2}(Y; \Omega^1_Y)$ is surjective. 
\end{lemma}
\begin{proof} Let $\mathcal{C}^\ccdot_{\mathcal{Y}}$ be the complex
$\{\scrO_{\mathcal{Y}}  \xrightarrow{d} \Omega^1_{\mathcal{Y}/\Spec A}\}$. Since $H^{n-1}(Y; \scrO_Y) =0$, a standard argument shows that $H^{n-1}(\mathcal{Y};\scrO_{\mathcal{Y}}) =0$ as well. Arguing as in the proof of Lemma~\ref{lemma0.3}(i),  there is a surjection 
$$H^{n-2}(\mathcal{Y};  \Omega^1_{\mathcal{Y}/\Spec A}) \to \mathbb{H}^{n-1}(\mathcal{Y}; \mathcal{C}^\ccdot_{\mathcal{Y}}) .$$
From the morphisms of complexes $\Cee[0] \to A[0] \to \mathcal{C}^\ccdot_{\mathcal{Y}}$, there is a commutative diagram
$$\begin{CD}
H^{n-1}(Y;\Cee) @>>> \mathbb{H}^{n-1}(\mathcal{Y}; \mathcal{C}^\ccdot_{\mathcal{Y}}) \\
@V{=}VV @VVV \\
H^{n-1}(Y;\Cee) @>>> \mathbb{H}^{n-1}(Y;\mathcal{C}^\ccdot_Y) \cong H^{n-2}(Y; \Omega^1_Y).
\end{CD}$$
By Lemma~\ref{lemma0.3}(iii), the map  $H^{n-1}(Y;\Cee) \to \mathbb{H}^{n-1}(Y;\mathcal{C}^\ccdot_Y)$ is surjective. Since it factors through the map $\mathbb{H}^{n-1}(\mathcal{Y}; \mathcal{C}^\ccdot_{\mathcal{Y}}) \to \mathbb{H}^{n-1}(Y;\mathcal{C}^\ccdot_Y)$, this last homomorphism is surjective as well. Then there is a sequence of surjections
$$H^{n-2}(\mathcal{Y};  \Omega^1_{\mathcal{Y}/\Spec A}) \to \mathbb{H}^{n-1}(\mathcal{Y}; \mathcal{C}^\ccdot_{\mathcal{Y}}) \to \mathbb{H}^{n-1}(Y;\mathcal{C}^\ccdot_Y) \cong H^{n-2}(Y; \Omega^1_Y).$$
Thus $H^{n-2}(\mathcal{Y}; \Omega^1_{\mathcal{Y}/\Spec A}) \to H^{n-2}(Y; \Omega^1_Y)$ is surjective. 
\end{proof} 

Finally, the fact that Lemma~\ref{lemma0.4} implies Theorem~\ref{surjthm} is a formal argument \cite[III 12.10 and 12.5]{Hartshorne} (cf.\ also the first paragraph of the proof of \cite[Theorem 4.1]{FL22c}), namely that for a surjection $A\to A'$, the natural map
$$H^{n-2}(\mathcal{Y}; \Omega^1_{\mathcal{Y}/\Spec A})\otimes _AA' \to H^{n-2}(\mathcal{Y}; \Omega^1_{\mathcal{Y}/\Spec A}\otimes_AA') \cong H^{n-2}(\mathcal{Y}'; \Omega^1_{\mathcal{Y}'/\Spec A'})$$
is surjective. (Here, we use the lci assumption to show the flatness of $ \Omega^1_{\mathcal{Y}/\Spec A}$ as an $A$-module, which is an easy consequence of the local criterion of flatness (cf.\ \cite[proof of Theorem 2.5]{FL22c}), but it is enough as noted in the proof of  \cite[Theorem 2.2]{gross_defcy}  that it is flat away from a zero-dimensional subspace.)
\end{proof} 

\subsection{The non lci case}  If we want to extend these results to the non lci case, then, following  Gross \cite[Lemma 2.4]{gross_defcy} and \cite[\S8 and Corollary 13.1]{Imagi}, we must consider the cokernel of the map $H_Z^{n-2}(Y;  \Omega^1_Y) \to H^{n-2}(Y;  \Omega^1_Y)$. (In the lci case, $H^{n-2}_Z(Y; \Omega^1_Y)$ is easily seen to be $0$.)    The main point is to show the following generalization of \cite[(13.2)]{Imagi}:
 
\begin{theorem}\label{surjthm2} Suppose  that
\begin{enumerate}
\item[\rm(i)] The $\partial\bar\partial$-lemma holds for $\hY$.
\item[\rm(ii)] The singularities of $Y$ are isolated and  Du Bois.
\item[\rm(iii)]  $H^{n-1}(Y; \scrO_Y) = 0$.
\end{enumerate}
Let $A$ be an Artin local $\Cee$-algebra  and let $\mathcal{Y} \to \Spec A$ be a   deformation of $Y$ over $\Spec A$. Then  the natural homomorphism $H^{n-2}(\mathcal{Y}; \Omega^1_{\mathcal{Y}/\Spec A})\to H^{n-2}(Y; \Omega^1_Y)/ \im H_Z^{n-2}(Y;  \Omega^1_Y)$   is surjective.
\end{theorem} 
 \begin{proof}
First assume that $n\ge 4$.  The lci assumption is used in two places in the proof of Theorem~\ref{surjthm}: Firstly, it implies  that $\Omega^1_Y$ is reflexive and hence  that $\pi_*\Omega^1_{\hY}(\log E)(-E) =\Omega^1_Y$.   In the general case, we have the following result for $n\ge 4$:
 
 \begin{lemma}\label{lemma4.1} For $n\ge 4$, the natural map $ H^{n-2}(Y;  \Omega^1_Y)  \to  H^{n-2}(Y; \pi_*\Omega^1_{\hY}(\log E)(-E) )$ is an isomorphism.
 \end{lemma} 
 \begin{proof} This follows immediately from the fact that the kernel and cokernel of the morphism $ \Omega^1_Y \to \pi_*\Omega^1_{\hY}(\log E)(-E)$ are supported on a finite set.
 \end{proof}

 In the second place,  for $n \ge 4$, we used the lci assumption to show that $R^{n-3}\pi_*\Omega^1_{\hY}(\log E)(-E) = 0$. In general, we have the following:
 
 \begin{lemma}\label{lemma4.2}  For $n\ge 4$, let $R$ be the image of $H^0(Y; R^{n-3}\pi_*\Omega^1_{\hY}(\log E)(-E))$ in $H^{n-2}(Y;  \Omega^1_Y)$ via the composition of the map $H^0(Y; R^{n-3}\pi_*\Omega^1_{\hY}(\log E)(-E)) \to  H^{n-2}(Y; \pi_*\Omega^1_{\hY}(\log E)(-E) )$ and the inverse of the isomorphism $ H^{n-2}(Y;  \Omega^1_Y)  \to  H^{n-2}(Y; \pi_*\Omega^1_{\hY}(\log E)(-E) )$ of Lemma~\ref{lemma4.1}. Then $R$ is contained in  the image of the map $H_Z^{n-2}(Y;  \Omega^1_Y) \to H^{n-2}(Y;  \Omega^1_Y)$. 
 \end{lemma}
  \begin{proof}    Let $\mathcal{F}$ be an arbitrary  sheaf on $\hY$.
 The Leray spectral sequences in local and ordinary cohomology are compatible, so there is a commutative diagram
 $$\begin{CD}
 H^0_Z(Y; R^{i-1}\pi_*\mathcal{F}) @>>> H^i_Z(Y; R^0\pi_*\mathcal{F})\\
 @VVV @VVV\\
  H^0 (Y;R^{i-1}\pi_*\mathcal{F}) @>>> H^i(Y; R^0\pi_*\mathcal{F}).
  \end{CD}$$
  For $i \ge 2$, $R^{i-1}\pi_*\mathcal{F}$ is supported on $Z$ and hence $H^0_Z(Y; R^{i-1}\pi_*\mathcal{F}) \cong H^0(Y; R^{i-1}\pi_*\mathcal{F})$. Thus, for $i \ge 2$,  the image of $H^0 (Y;R^{i-1}\pi_*\mathcal{F})$ in $ H^i(Y; R^0\pi_*\mathcal{F})$ is contained in the image of $H^i_Z(Y; R^0\pi_*\mathcal{F})$. Applying this to $\mathcal{F} = \Omega^1_{\hY}(\log E)(-E)$, since $n \ge 4$, the image    of the map 
  $$H^0(Y; R^{n-3}\pi_*\Omega^1_{\hY}(\log E)(-E))\to H^{n-2}(Y; \pi_*\Omega^1_{\hY}(\log E)(-E) )$$ 
  is contained in the image of $H^{n-2}_Z(Y; \pi_*\Omega^1_{\hY}(\log E)(-E) )$. 
Since the kernel and cokernel of the morphism $ \Omega^1_Y \to \pi_*\Omega^1_{\hY}(\log E)(-E)$ are supported on a finite set and $n \ge 4$, $H_Z^{n-2}(Y;  \Omega^1_Y)\cong  H^{n-2}_Z(Y; \pi_*\Omega^1_{\hY}(\log E)(-E))$. The commutative diagram   
$$\begin{CD}
  H_Z^{n-2}(Y;  \Omega^1_Y) @>>>  H^{n-2}(Y;  \Omega^1_Y)\\
  @V{\cong}VV @VV{\cong}V \\
  H^{n-2}_Z(Y; \pi_*\Omega^1_{\hY}(\log E)(-E))  @>>>  H^{n-2}(Y; \pi_*\Omega^1_{\hY}(\log E)(-E) ) 
  \end{CD}$$
  then implies that $R$ is contained in  the image of the map $H_Z^{n-2}(Y;  \Omega^1_Y) \to H^{n-2}(Y;  \Omega^1_Y)$. 
 \end{proof}
 
 Next, the Hodge theory methods of this section  show the following:
 
  \begin{lemma}\label{lemma4.3}   For $n\ge 4$, let $R$ be as in Lemma~\ref{lemma4.2} and let $\mathcal{C}^\ccdot_Y$ be the complex $\{\scrO_Y  \xrightarrow{d} \Omega^1_Y\}$.     Then 
  \begin{enumerate}
  \item[\rm(i)] The composed map 
 \begin{gather*}
 H^{n-2}(Y; \Omega^1_Y)/R \cong H^{n-2}(Y; \pi_*\Omega^1_{\hY}(\log E)(-E) )/\im  H^0(Y; R^{n-3}\pi_*\Omega^1_{\hY}(\log E)(-E))\\
 \to H^{n-2}(\hY; \Omega^1_{\hY}(\log E)(-E))
 \end{gather*}
  is injective.
 \item[\rm(ii)]  The image of the differential $d\colon H^{n-2}(Y; \scrO_Y) \to H^{n-2}(Y; \Omega^1_Y)$ is contained in $R$, and hence the induced map $H^{n-2}(Y; \scrO_Y) \to H^{n-2}(Y; \Omega^1_Y)/R$ is zero.
\item[\rm(iii)] The inclusion $\Omega^1_Y[-1] \to \mathcal{C}^\ccdot_Y$ induces a surjection $\mathbb{H}^{n-1}(Y;\mathcal{C}^\ccdot_Y) \to H^{n-2}(Y; \Omega^1_Y)/R$.
\item[\rm(iv)] The homomorphism
$\Cee \to \mathcal{C}^\ccdot_Y$ induces a surjection $H^{n-1}(Y; \Cee) \to \mathbb{H}^{n-1}(Y;\mathcal{C}^\ccdot_Y)$ and hence the composition
$$H^{n-1}(Y; \Cee) \to \mathbb{H}^{n-1}(Y;\mathcal{C}^\ccdot_Y) \to H^{n-2}(Y; \Omega^1_Y)/R$$
is surjective.
\end{enumerate}
 \end{lemma}
  \begin{proof}  The proofs of these statements are minor variations of the proofs of Lemmas~\ref{lemma0.2}  and \ref{lemma0.3}.
 \end{proof}

 To complete the proof of Theorem~\ref{surjthm2}, define   the complex
$\mathcal{C}^\ccdot_{\mathcal{Y}}= \{\scrO_{\mathcal{Y}}  \xrightarrow{d} \Omega^1_{\mathcal{Y}/\Spec A}\}$ as in Lemma~\ref{lemma0.4}. The proof there  shows that the map
$H^{n-2}(\mathcal{Y}; \Omega^1_{\mathcal{Y}/\Spec A})\to H^{n-2}(Y; \Omega^1_Y)/R$   is surjective. By Lemma~\ref{lemma4.2}, $R\subseteq \im H_Z^{n-2}(Y;  \Omega^1_Y) $. Thus there is a well-defined surjection 
$$H^{n-2}(Y; \Omega^1_Y)/R \to H^{n-2}(Y; \Omega^1_Y)/\im H_Z^{n-2}(Y;  \Omega^1_Y) .$$
Then the composition $H^{n-2}(\mathcal{Y}; \Omega^1_{\mathcal{Y}/\Spec A})\to H^{n-2}(Y; \Omega^1_Y)/R \to H^{n-2}(Y; \Omega^1_Y)/ \im H_Z^{n-2}(Y;  \Omega^1_Y)$   is surjective as well.

 The above proof works with minor modifications in case $n=3$.  In this case, by the Leray spectral sequence, the natural map
$$H^1(Y; \pi_*\Omega^1_{\hY}(\log E)(-E)) \to H^1(\hY; \Omega^1_{\hY}(\log E)(-E))$$ 
is injective. There are two  short exact sequences
\begin{align*}
0 \to F \to &\Omega^1_Y \to \mathcal{S} \to 0;\\
0\to \mathcal{S}  \to   \pi_*\Omega^1_{\hY}(\log &E)(-E) \to F' \to 0,\\
\end{align*}
where $F$ and $F'$ are  supported on $Z$. Thus $H^1(Y; \Omega^1_Y) \cong H^1(Y;\mathcal{S})$ and $H^1_Z(Y; \Omega^1_Y) \cong H^1_Z(Y;\mathcal{S})$, there is an exact sequence
$$H^0(Y;F') \to H^1(Y;\mathcal{S}) \to  H^1(Y; \pi_*\Omega^1_{\hY}(\log E)(-E)) \to 0,$$
and the following diagram commutes:
$$\begin{CD}
H^0_Z(Y; F') @>>> H^1_Z(Y;\mathcal{S}) \cong H^1_Z(Y; \Omega^1_Y)\\
@V{\cong}VV @VVV   \\
H^0 (Y; F') @>>> H^1(Y;\mathcal{S}) \cong H^1(Y; \Omega^1_Y).
\end{CD}$$
In particular, identifying $ H^1(Y;\mathcal{S})$ with $H^1(Y; \Omega^1_Y)$,  let $R$ be  the image of $H^0 (Y; F')$ in $H^1(Y; \Omega^1_Y)$. Then there is an induced injection $H^1(Y; \Omega^1_Y)/R \to  H^1(\hY; \Omega^1_{\hY}(\log E)(-E))$, and $R$  is contained in the image of $H^1_Z(Y; \Omega^1_Y)$. The rest of the argument follows the lines of the proof in case $n \ge 4$.
\end{proof}

\subsection{The case of a pair}\label{subsection2.4}  In this section, we generalize the previous results to the case of a pair $(Y,D)$. Here we keep the previous conventions on  $Y$  and let  $D=\bigcup_iD_i\subseteq Y$ be  a divisor with simple normal crossings contained in the smooth locus of $Y$. Thus we can identify $D$ with its preimage in $\hY$. We then have a spectral sequence with $E_1$ term 
\begin{align*}
E_1^{p,q} &= H^q(\hY; \Omega^p_{\hY}(\log(E+D)(-E-D)) \implies \mathbb{H}^{p+q}(\hY;  \Omega^\ccdot_{\hY}(\log(E+D)(-E-D)) \\
&\cong H^{p+q}(\hY, E\cup D;\Cee) \cong H^{p+q}(Y,Z\cup D ; \Cee),
\end{align*}
which degenerates at the $E_1$ page if the $\partial\bar\partial$-lemma holds for $\hY$ and all of the $k$-fold intersections $D_{i_1}\cap\cdots \cap D_{i_k}$. Finally,  a \textsl{log deformation of the  pair} $(Y,D)$ over an Artin local $\Cee$-algebra $A$  is a pair $(\mathcal{Y}, \mathcal{D})$, where $\mathcal{D} = \bigcup_i\mathcal{D}_i$ is a  Cartier divisor in $\mathcal{Y}$,  together with a flat proper morphism $\mathcal{Y} \to \Spec A$ and an isomorphism $\mathcal{Y}\times_{\Spec A} \Spec \Cee \to Y$ such that each $\mathcal{D}_i$ is smooth over $\Spec A$ and the morphism $\mathcal{Y}\times_{\Spec A} \Spec \Cee \to Y$ identifies $\mathcal{D}_i\times_{\Spec A} \Spec \Cee$ with $D_i$. The complex $\Omega^\ccdot_{\mathcal{Y}/\Spec A}(\log\mathcal{D})(-\mathcal{D})$ is defined in the usual way.

Minor modifications of the arguments above show:

\begin{theorem}\label{pairsurjthm} Suppose that
\begin{enumerate}
\item[\rm(i)] The $\partial\bar\partial$-lemma holds for $\hY$ and all of the $k$-fold intersections $D_{i_1}\cap\cdots \cap D_{i_k}$.
\item[\rm(ii)] The singularities of $Y$ are isolated, lci and Du Bois.
\item[\rm(iii)] $H^{n-1}(Y; \scrO_Y(-D)) = 0$.
\end{enumerate}
Let $A$ be an Artin local $\Cee$-algebra  and let $(\mathcal{Y},\mathcal{D})\to \Spec A$ be a  log deformation of the  pair $(Y,D)$ over $\Spec A$. Then  the natural homomorphism 
$$H^{n-2}(\mathcal{Y}; \Omega^1_{\mathcal{Y}/\Spec A}(\log\mathcal{D})(-\mathcal{D}))\to H^{n-2}(Y; \Omega^1_Y(\log D)(-D))$$ is surjective.

Without  the assumption in (ii) that the singularities are local complete intersections, then, with $A$ and $(\mathcal{Y},\mathcal{D})\to \Spec A$ as above, 
the natural homomorphism 
$$H^{n-2}(\mathcal{Y}; \Omega^1_{\mathcal{Y}/\Spec A}(\log\mathcal{D})(-\mathcal{D}))\to H^{n-2}(Y; \Omega^1_Y(\log D)(-D))/H^{n-2}_Z(Y; \Omega^1_Y(\log D)(-D))$$ is surjective.
\end{theorem} 
\begin{proof} The proofs of Theorems~\ref{surjthm} and ~\ref{surjthm2} can be easily modified to deal with this case: replace $\scrO_Y$ and $I_Z$ by $\scrO_Y(-D)$ and $I_Z(-D)$, $\Omega^1_Y$ by $\Omega^1_Y(\log D)(-D)$, $\scrO_{\hY}(-E)$ by $\scrO_{\hY}(-E-D)$ and $\Omega^1_{\hY}(\log E)(-E)$ by $\Omega^1_{\hY}(\log (E+D))(-E-D)$ throughout.
The sheaves $i_!\Cee_{Y-Z}$ and $\Cee$ (the constant sheaf on $Y$)    are then replaced by $j_!\Cee_{Y-Z-D}$ and $j'_!\Cee_{Y-D}$ respectively, where $j\colon Y-Z-D \to Y$ and $j'\colon Y-D \to Y$ are the inclusions. Finally, the complex $\mathcal{C}^\ccdot_{\mathcal{Y}}$ is replaced by the complex
$$\{\scrO_{\mathcal{Y}}(-\mathcal{D})  \xrightarrow{d} \Omega^1_{\mathcal{Y}/\Spec A}(\log\mathcal{D})(-\mathcal{D})\}.$$
The proofs are then essentially unchanged.
\end{proof}

\section{Results from deformation theory} \label{section3} 

The goal of this mostly expository section is to prove Theorem~\ref{mainthm} along the lines of Namikawa \cite[Proof of Theorem 1]{namtop}, as well as   corresponding generalizations to the non lci case (Theorem~\ref{genT1thm})   along the lines of  Gross \cite[Theorem 2.2]{gross_defcy} and to pairs (Theorem~\ref{pairmainthm}). \textbf{Throughout this section, unless otherwise specified, we assume that $Y$ is a  singular Calabi-Yau variety of dimension $n\ge 3$ with isolated singularities and  that there exists a resolution $\pi\colon \hY \to Y$ satisfying the $\partial\bar\partial$-lemma.}   Before we begin, we note the following two lemmas: 

\begin{lemma}\label{lemma3.5}  If $Y$ has at worst   Du Bois singularities and $\mathcal{Y} \to \Spec A$ is a deformation of $Y$ over the Artin local ring $A$, then $\omega_{\mathcal{Y}/\Spec A} \cong \scrO_{\mathcal{Y}}$. 
\end{lemma} 
\begin{proof} This  argument is sketched in  \cite[Proof of Corollary 1.5]{FL22c}.  (There is also another, somewhat easier, proof in case  $H^1(Y; \scrO_Y) = 0$.)
\end{proof}

\begin{lemma}[Namikawa]\label{Namlemma}  If $Y$ has at worst  isolated rational  singularities and $H^1(Y; \scrO_Y) = 0$, then $H^0(Y; T^0_Y) =0$.
\end{lemma}  
\begin{proof}  We outline the proof given in \cite[p.\ 12]{NamikawaSugaku}. Choose an equivariant resolution $\pi\colon \hY \to Y$, i.e.\ a resolution such that $\pi_*T_{\hY} = T^0_Y$.  Since rational Gorenstein singularities are canonical, there exists an inclusion $T_{\hY} \to \Omega^{n-1}_{\hY}$. Thus it suffices to prove that $H^0(\hY;  \Omega^{n-1}_{\hY})=0$, i.e.\ that $h^{n-1, 0}(\hY) = 0$. By Hodge symmetry, 
$$h^{n-1, 0}(\hY) = h^{0, n-1}(\hY) = \dim H^{n-1}(\hY; \scrO_{\hY}) = \dim H^{n-1}(Y; \scrO_Y),$$
where we have used the Leray spectral sequence and the assumption of rational singularities to get the last inequality. By Serre duality on $Y$, $\dim H^{n-1}(Y; \scrO_Y) = \dim H^1(Y; \scrO_Y)$  which is $0$ by assumption. Hence $H^0(Y; T^0_Y) =0$.
\end{proof}

\subsection{The $T^1$ lifting property}   Let $Y$ be a reduced Cohen-Macaulay compact analytic space and let $\mathbf{Def}_Y$  denote the deformation functor from the category of Artin local $\Cee$-algebras to sets.  Let $A_k =\Cee[t]/(t^{k+1})$ and   let $B_k = A_k\otimes_\Cee A_1=\Cee[t,s]/(t^{k+1},s^2)$. There are $\Cee$-algebra homomorphisms $B_k \to A_k$,  $A_k \to A_{k-1}$ and $B_k \to B_{k-1}$.   Let $\psi_k\colon \mathcal{Y}_k\to \Spec A_k$ be a deformation of $Y$ over $\Spec A_k$ and let $\mathcal{Y}_{k-1}= \mathcal{Y}_k\times_{\Spec A_k}\Spec A_{k-1}$. Define
$$T^1(\mathcal{Y}_k/A_k) =\{((\mathcal{Y}, \psi): \mathcal{Y} \in \mathbf{Def}_Y(B_k) , \psi\colon  \mathcal{Y}\times _{\Spec B_k}\Spec A_k \to \mathcal{Y}_k \text{ is an isomorphism} \}.$$
The homomorphism $B_k \to B_{k-1}$ defines a  function $T^1(\mathcal{Y}_k/A_k)\to  T^1(\mathcal{Y}_{k-1}/A_{k-1})$.  Then the   $T^1$ lifting criterion  \cite{Kawamata}, Gross \cite[Theorem 1.14]{gross_defcy}, \cite{FanMan} says the following:

\begin{theorem}\label{theorem3.1}\label{T1}  If the natural function $$\mathbf{Def}_Y(B_k) \to \mathbf{Def}_Y(A_k)\times_{\mathbf{Def}_Y(A_{k-1})}\mathbf{Def}_Y(B_{k-1})$$ 
is surjective,  then $\mathbf{Def}_Y$ is unobstructed. Moreover, if the map $T^1(\mathcal{Y}_k/A_k)\to  T^1(\mathcal{Y}_{k-1}/A_{k-1})$ is surjective for every $k$ and every choice of $\mathcal{Y}_k\in \mathbf{Def}_Y(A_k)$, then $\mathbf{Def}_Y$ is unobstructed.
 \qed
\end{theorem}  

To check that the $T^1$ lifting criterion holds, we have the following result of \cite{Kawamata}:

\begin{lemma}\label{lemma3.2}    $T^1(\mathcal{Y}_k/A_k)\cong \Ext^1_{\scrO_{\mathcal{Y}_k}}( \Omega^1_{\mathcal{Y}_k/\Spec A_k}, \scrO_{\mathcal{Y}_k})$.  
\end{lemma}
\begin{proof} Given $\mathcal{Y} \in \mathbf{Def}_Y(B_k)$ such that  $\mathcal{Y}\times _{\Spec B_k}\Spec A_k \cong \mathcal{Y}_k$, we can view $\mathcal{Y}$ as a deformation of $\mathcal{Y}_k$ over $\Spec A_1 = \Spec \Cee[\varepsilon]$, where $\varepsilon^2 =0$. Then the conormal sequence for the closed subspace $\mathcal{Y}_k$ of $\mathcal{Y}$  is an exact sequence
$$  \scrO_{\mathcal{Y}_k}  \to \Omega^1_{\mathcal{Y}/\Spec B_k}|\mathcal{Y}_k \to \Omega^1_{\mathcal{Y}_k/\Spec A_k} \to 0.$$
Since  $Y$ is   assumed to be reduced and Cohen-Macaulay, the map $ \scrO_{\mathcal{Y}_k}  \to \Omega^1_{\mathcal{Y}/\Spec B_k}|\mathcal{Y}_k$ is injective since it is injective on a dense open subset. Thus $\mathcal{Y}$ defines an extension of  $\Omega^1_{\mathcal{Y}_k/\Spec A_k}$ by  $\scrO_{\mathcal{Y}_k}$, i.e.\ a class in $\Ext^1_{\scrO_{\mathcal{Y}_k}}( \Omega^1_{\mathcal{Y}_k/\Spec A_k}, \scrO_{\mathcal{Y}_k})$.  Standard arguments (cf.\ \cite[Theorem 1.1.10]{Sernesi} or \cite[(2.6)]{NamikawaSugaku}) show that this sets up a bijection from 
$T^1(\mathcal{Y}_k/A_k)$ to  $\Ext^1_{\scrO_{\mathcal{Y}_k}}( \Omega^1_{\mathcal{Y}_k/\Spec A_k}, \scrO_{\mathcal{Y}_k})$.
\end{proof} 

Given an extension $\mathcal{E}$ of $\Omega^1_{\mathcal{Y}_k/\Spec A_k}$ by  $\scrO_{\mathcal{Y}_k}$, taking the tensor product with $\scrO_{\mathcal{Y}_{k-1}}$ gives an exact sequence
$$\scrO_{\mathcal{Y}_{k-1}} \to \mathcal{E} \otimes_{\scrO_{\mathcal{Y}_k}} \scrO_{\mathcal{Y}_{k-1}} \to  \Omega^1_{\mathcal{Y}_k/\Spec A_k} \otimes_{\scrO_{\mathcal{Y}_k}} \scrO_{\mathcal{Y}_{k-1}} \to 0.$$
By base change, $ \Omega^1_{\mathcal{Y}_k/\Spec A_k} \otimes_{\scrO_{\mathcal{Y}_k}} \scrO_{\mathcal{Y}_{k-1}} \cong  \Omega^1_{\mathcal{Y}_{k-1}/\Spec A_{k-1}}$, and the homomorphism $\scrO_{\mathcal{Y}_{k-1}} \to \mathcal{E} \otimes_{\scrO_{\mathcal{Y}_k}} \scrO_{\mathcal{Y}_{k-1}}$ is injective since it is injective on a dense open subset. Thus the sequence 
$$0\to \scrO_{\mathcal{Y}_{k-1}} \to \mathcal{E} \otimes_{\scrO_{\mathcal{Y}_k}} \scrO_{\mathcal{Y}_{k-1}} \to  \Omega^1_{\mathcal{Y}_k/\Spec A_k} \otimes_{\scrO_{\mathcal{Y}_k}} \scrO_{\mathcal{Y}_{k-1}} \to 0$$ is exact. This defines a  homomorphism 
$$ \Ext^1_{\scrO_{\mathcal{Y}_k}}( \Omega^1_{\mathcal{Y}_k/\Spec A_k}, \scrO_{\mathcal{Y}_k}) \to  \Ext^1_{\scrO_{\mathcal{Y}_{k-1}}}( \Omega^1_{\mathcal{Y}_{k-1}/\Spec A_{k-1}}, \scrO_{\mathcal{Y}_{k-1}}).$$ More precisely, we have the following \cite[Lemma 2.3]{gross_defcy}:

\begin{lemma}\label{Grosslemma2.3}  If in addition $Y$ has isolated singularities, there are isomorphisms  
\begin{align*} \alpha_1\colon \Ext^1_{\scrO_{\mathcal{Y}_k}}( \Omega^1_{\mathcal{Y}_k/\Spec A_k}, \scrO_{\mathcal{Y}_{k-1}}) &\xrightarrow{\cong}   \Ext^1_{\scrO_{\mathcal{Y}_{k-1}}}( \Omega^1_{\mathcal{Y}_{k-1}/\Spec A_{k-1}}, \scrO_{\mathcal{Y}_{k-1}});\\
 \alpha_2\colon \Ext^2_{\scrO_{\mathcal{Y}_k}}( \Omega^1_{\mathcal{Y}_k/\Spec A_k}, \scrO_Y) &\xrightarrow{\cong}    \Ext^2_{\scrO_Y}( \Omega^1_Y, \scrO_Y). \qed
\end{align*} 
\end{lemma}

Unwinding the definitions shows the following:

\begin{lemma}\label{lemma3.3} Via the isomorphisms of Lemma~\ref{lemma3.2}, the diagram
$$\begin{CD}
T^1(\mathcal{Y}_k/A_k)@>>>  T^1(\mathcal{Y}_{k-1}/A_{k-1}) \\
@V{\cong}VV @VV{\cong}V\\
 \Ext^1_{\scrO_{\mathcal{Y}_k}}( \Omega^1_{\mathcal{Y}_k/\Spec A_k}, \scrO_{\mathcal{Y}_k}) @>>>  \Ext^1_{\scrO_{\mathcal{Y}_{k-1}}}( \Omega^1_{\mathcal{Y}_{k-1}/\Spec A_{k-1}}, \scrO_{\mathcal{Y}_{k-1}}),
\end{CD}$$
is commutative.  Thus   there is an exact sequence
$$T^1(\mathcal{Y}_k/A_k)\to T^1(\mathcal{Y}_{k-1}/A_{k-1})\xrightarrow{\delta}  \mathbb{T}^2_Y,$$
where $\delta$ is identified via Lemma~\ref{Grosslemma2.3} with the coboundary map associated to the long exact Ext sequence arising from 
$$0\to \scrO_Y \to  \scrO_{\mathcal{Y}_k} \to \scrO_{\mathcal{Y}_{k-1}} \to 0. \qed$$
\end{lemma} 

At this point, it is straightforward to prove Theorem~\ref{mainthm}  via   Namikawa's method \cite[Proof of Theorem 1]{namtop}, \cite{NamikawaSugaku}, under the added hypothesis that $H^0(Y; T^0_Y) = 0$ (which by Lemma~\ref{Namlemma}  is automatic in case $Y$ has rational singularities). In this case, $H^n(Y; \Omega^1_Y)$ is Serre dual to $\Hom(\Omega^1_Y, \scrO_Y) = H^0(Y; T^0_Y) = 0$.  In particular, given a deformation $\psi_k\colon \mathcal{Y}_k\to \Spec A_k$ of $Y$ over $\Spec A_k$, $R^n\psi_k{}_*\Omega^1_{\mathcal{Y}_k/\Spec A_k} =0$. Then by  base change \cite[III 12.11]{Hartshorne}, the map $$R^{n-1}\psi_k{}_*\Omega^1_{\mathcal{Y}_k/\Spec A_k}\to H^{n-1}(Y; \Omega^1_Y)$$ is surjective. By Lemma~\ref{lemma0.4}, the map $R^{n-2}\psi_k{}_*\Omega^1_{\mathcal{Y}_k/\Spec A_k}\to H^{n-2}(Y; \Omega^1_Y)$  is surjective. Again by base change,  $R^{n-1}\psi_k{}_*\Omega^1_{\mathcal{Y}_k/\Spec A_k}= H^{n-1}(\mathcal{Y}_k; \Omega^1_{\mathcal{Y}_k/\Spec A_k})$ is locally free and compatible with base change, i.e.\ the map 
$$H^{n-1}(\mathcal{Y}_k; \Omega^1_{\mathcal{Y}_k/\Spec A_k}) \otimes _{A_k} A_{k-1} \to H^{n-1}(\mathcal{Y}_{k-1}; \Omega^1_{\mathcal{Y}_{k-1}/\Spec A_{k-1}})$$
is an isomorphism. The argument of \cite[p.\ 13, (2.17) and (9)]{NamikawaSugaku} then shows that 
$$ \Ext^1_{\scrO_{\mathcal{Y}_k}}( \Omega^1_{\mathcal{Y}_k/\Spec A_k}, \scrO_{\mathcal{Y}_k}) \otimes _{A_k}A_{k-1} \to \Ext^1_{\scrO_{\mathcal{Y}_{k-1}}}( \Omega^1_{\mathcal{Y}_{k-1}/\Spec A_{k-1}}, \scrO_{\mathcal{Y}_{k-1}})$$ is an isomorphism, so that $T^1$ lifting holds. 

\subsection{Generalized  $T^1$ lifting}\label{subsection3.2}    For the moment, let  $Y$ be a compact analytic space with isolated Cohen-Macaulay singularities, let $\mathbf{Def}_Y$ be the deformation functor for $Y$ and let $\mathbf{Def}_{Y, Z}$ be the local deformation functor for the singular points of $Y$. Thus there is  a morphism $\mathbf{Def}_Y \to \mathbf{Def}_{Y, Z}$, which at the level of Zariski tangent spaces is the map $\mathbb{T}^1_Y \to H^0(Y; T^1_Y) =\bigoplus_{x\in Z}H^0(T^1_{Y,x})$. 

\begin{definition}\label{defgenT1}    Let $\delta$ be  the map defined in Lemma~\ref{lemma3.3}  and let $\ell \colon  \mathbb{T}^2_Y \cong \Ext^2_{\scrO_Y}(\Omega^1_Y, \scrO_Y) \to H^0(Y;\mathit{Ext}^2_{\scrO_Y}(\Omega^1_Y, \scrO_Y)  )$ be the natural map.  The \textsl{generalized $T^1$ lifting property}  holds for $Y$ if
the restriction of $\ell$ to $\im \delta$ is injective. Note that, in the lci case, $H^0(Y;\mathit{Ext}^2_{\scrO_Y}(\Omega^1_Y, \scrO_Y)  )=0$ so that the generalized $T^1$ lifting property reduces to the $T^1$ lifting criterion.
\end{definition} 

The following is then due to  \cite[Theorem 1.16]{gross_defcy}. For related versions, see also    \cite{Kawamata2}, \cite{Kawamata-E}, \cite{Ran0}, and   \cite[Theorem 2.2]{FanMan}.

\begin{theorem}\label{conseqgenT1}   Suppose that $\mathbf{Def}_Y$ is prorepresentable and the generalized $T^1$ lifting property holds for $Y$. 
Then 
\begin{enumerate}
\item[\rm(i)]  If  the map $\mathbb{T}^1_Y \to H^0(Y; T^1_Y)$ is surjective, then the morphism $\mathbf{Def}_Y \to \mathbf{Def}_{Y, Z}$ is smooth. In particular, if $(S,0)$ is a germ which is a   locally semiuniversal  deformation of $\mathbf{Def}_{Y, Z}$, then the base of the  locally semiuniversal  deformation of $Y$ is isomorphic to the germ  $(H^1(Y;T^0_Y)\times S,0)$. 
\item[\rm(ii)] In general, if the dimension of the cokernel of the map $\mathbb{T}^1_Y \to H^0(Y; T^1_Y)$ is $r$, then there exist $r$ functions $f_1, \dots, f_r\in \mathfrak{m}$, where $\mathfrak{m}$ is the maximal ideal of $0$ in  the local ring $\scrO_{H^1(Y;T^0_Y)\times S,0}$, such that the germ of the subspace defined by $f_1, \dots, f_r$ has the same underlying reduced subspace as that of the germ of the base of the semiuniversal deformation of $Y$.  \qed
\end{enumerate}
\end{theorem}

 We return now to the assumption that $Y$ is a singular Calabi-Yau variety of dimension $n\ge 3$ with isolated singularities. Before we state the main result that we need, we introduce the following notation. For an $A_k$-module $M$, we write $M\spcheck = \Hom_{A_k}(M, A_k)$. If  $M$ is actually an $A_{k-1}$-module, i.e.\ if $M$ is annihilated by $t^k$, then  $M\spcheck \cong \Hom_{A_{k-1}}(M, A_{k-1})$,   viewing $A_{k-1}$ as the submodule $tA_k$ of $A_k$, and similarly if $M$ is annihilated by $t$. Since $A_k$ is Gorenstein, the functor $M \mapsto M\spcheck$ is exact.

 With that said, in order to apply Theorem~\ref{conseqgenT1}, the main point is the following result of Gross (\cite[Lemmas 2.3 and 2.4,   (2.16), and Remark 2.18]{gross_defcy}):

\begin{theorem}\label{proposition3.7}  Let $\alpha_i$, $i=1,2$  be as defined in Lemma~\ref{Grosslemma2.3} and let 
$$\delta'\colon \Ext^1_{\scrO_{\mathcal{Y}_k}}( \Omega^1_{\mathcal{Y}_k/\Spec A_k}, \scrO_{\mathcal{Y}_{k-1}})  \to  \Ext^2_{\scrO_{\mathcal{Y}_k}}( \Omega^1_{\mathcal{Y}_k/\Spec A_k}, \scrO_Y)$$
be the homomorphism in  the long exact Ext sequence coming from applying $\Ext^i_{\scrO_{\mathcal{Y}_k}}( \Omega^1_{\mathcal{Y}_k/\Spec A_k}, -)$ to the exact sequence
$$0 \to \scrO_Y \to \scrO_{\mathcal{Y}_k} \to  \scrO_{\mathcal{Y}_{k-1}} \to 0,$$
i.e.\ $\delta'$ corresponds to $\delta$ as defined in  Lemma~\ref{lemma3.3}.
\begin{enumerate} 
\item[\rm(i)] There is an isomorphism of  exact sequences, i.e.\ a commutative diagram with exact columns such that the horizontal maps are all isomorphisms:
$$\begin{CD}
   H^{n-1}(\mathcal{Y}_k; \Omega^1_{\mathcal{Y}_k/\Spec A_k})\spcheck  @>{\tau_0}>>     \Ext^1_{\scrO_{\mathcal{Y}_k}}( \Omega^1_{\mathcal{Y}_k/\Spec A_k}, \scrO_{\mathcal{Y}_k}) @>{=}>> T^1(\mathcal{Y}_k/A_k) \\
   @VVV @VVV @VVV \\
    H^{n-1}(\mathcal{Y}_{k-1}; \Omega^1_{\mathcal{Y}_{k-1}/\Spec A_{k-1}})\spcheck  @>{\tau_1}>> \Ext^1_{\scrO_{\mathcal{Y}_k}}( \Omega^1_{\mathcal{Y}_k/\Spec A_k}, \scrO_{\mathcal{Y}_{k-1}})  @>{\alpha_1}>> T^1(\mathcal{Y}_{k-1}/A_{k-1}) \\
    @VVV @VV{\delta'}V @VV{\delta}V \\
    H^{n-2}(Y; \Omega^1_Y)\spcheck   @>{\tau_2}>> \Ext^2_{\scrO_{\mathcal{Y}_k}}( \Omega^1_{\mathcal{Y}_k/\Spec A_k}, \scrO_Y) @>{\alpha_2}>>   \mathbb{T}^2_Y\\
    @VVV @VVV @. \\
   H^{n-2}(\mathcal{Y}_k; \Omega^1_{\mathcal{Y}_k/\Spec A_k})\spcheck @>{\cong}>> \Ext^2_{\scrO_{\mathcal{Y}_k}}( \Omega^1_{\mathcal{Y}_k/\Spec A_k}, \scrO_{\mathcal{Y}_k}) @.
\end{CD}$$
% @V{\tau_0}V{\cong}V   @V{\tau_1}V{\cong}V   @V{\tau_2}V{\cong}V @V{\cong}VV \\
%@| @V{\alpha_1}VV  @V{\alpha_2}VV  \\
  \item[\rm(ii)]  There is an isomorphism $\tau\colon H_Z^{n-2}(Y;  \Omega^1_Y)\spcheck \to H^0(Y; \mathit{Ext}^2_{\scrO_Y}(\Omega^1_Y, \scrO_Y))$ such that the following diagram commutes:
$$\begin{CD} 
 H^{n-2}(Y; \Omega^1_Y)\spcheck  @>>>  H^{n-2}_Z(Y; \Omega^1_Y)\spcheck \\
 @V{\alpha_2\circ \tau_2}VV  @V{\tau}V{\cong}V \\
 \Ext^2_{\scrO_Y}(\Omega^1_Y, \scrO_Y) @>{\ell}>> H^0(Y; \mathit{Ext}^2_{\scrO_Y}(\Omega^1_Y, \scrO_Y)). \qed
 \end{CD}$$
 \end{enumerate}
 \end{theorem}

 We can now give a  proof of   Theorem~\ref{mainthm} that does not need the assumption  $H^0(Y; T^0_Y) =0$:
 
 \begin{proof}[Proof of Theorem~\ref{mainthm}]   By Theorem~\ref{T1}, we must prove that the map $T^1(\mathcal{Y}_k/A_k) \to  T^1(\mathcal{Y}_{k-1}/A_{k-1})$ is surjective for every $k$ and every choice of $\mathcal{Y}_k$.  By Theorem~\ref{proposition3.7}, given $\mathcal{Y}_k$, we must show that the map  $H^{n-1}(\mathcal{Y}_{k-1}; \Omega^1_{\mathcal{Y}_{k-1}/\Spec A_{k-1}}) \to  H^{n-1}(\mathcal{Y}_k; \Omega^1_{\mathcal{Y}_k/\Spec A_k})$ is injective. It suffices to prove that the coboundary map $\partial \colon H^{n-2}(Y; \Omega^1_Y) \to H^{n-1}(\mathcal{Y}_{k-1}; \Omega^1_{\mathcal{Y}_{k-1}/\Spec A_{k-1}})$ is zero, or  equivalently   that the  map $H^{n-2}(\mathcal{Y}_k; \Omega^1_{\mathcal{Y}_k/\Spec A_k}) \to H^{n-2}(Y; \Omega^1_Y)$ is surjective. In the lci case, this is a consequence of Lemma~\ref{lemma0.4}. 
\end{proof}

In case $Y$ does not have lci singularities, there is the following:

\begin{theorem}\label{genT1thm}   Suppose that $Y$ is a singular Calabi-Yau variety of dimension $n\ge 3$ such that
\begin{enumerate}
\item[\rm(i)] The $\partial\bar\partial$-lemma holds for $\hY$.
\item[\rm(ii)] The singularities of $Y$ are isolated and Du Bois.
\item[\rm(iii)]  $H^1(Y; \scrO_Y) = 0$.
\end{enumerate}
Then $Y$ satisfies the generalized $T^1$ lifting property. 
\end{theorem} 
\begin{proof} By the  commutative diagram in Theorem~\ref{proposition3.7}(ii), the map $ H_Z^{n-2}(Y;  \Omega^1_Y) \to H^{n-2}(Y; \Omega^1_Y)$ is dual to the map 
$$\ell\colon \mathbb{T}^2_Y =\Ext^2_{\scrO_Y}(\Omega_Y, \scrO_Y) \to H^0(Y; T^2_Y) = H^0(Y; \mathit{Ext}^2_{\scrO_Y}(\Omega_Y, \scrO_Y)).$$
By  Theorem~\ref{surjthm2}, the map $H^{n-2}(\mathcal{Y}_k; \Omega^1_{\mathcal{Y}_k/\Spec A_k})\to H^{n-2}(Y; \Omega^1_Y)/ \im H_Z^{n-2}(Y;  \Omega^1_Y)$   is surjective. Then Theorem~\ref{proposition3.7} implies that the map $\ell|\im \delta$ of Definition~\ref{defgenT1} is injective, i.e.\ that the generalized $T^1$ lifting property holds. 
\end{proof} 

\begin{corollary} Suppose that $Y$ is a singular Calabi-Yau variety of dimension $n\ge 3$ such that
\begin{enumerate}
\item[\rm(i)] The $\partial\bar\partial$-lemma holds for $\hY$.
\item[\rm(ii)] The singularities of $Y$ are isolated and Du Bois.
\item[\rm(iii)]  $ H^1(Y; \scrO_Y) = 0$.
\item[\rm(iv)] The functor $\mathbf{Def}_Y$ is prorepresentable. (By Lemma~\ref{Namlemma}, this is automatic if the singularities are rational.) 
\end{enumerate}
Then   the conclusions of Theorem~\ref{conseqgenT1} hold for $\mathbf{Def}_Y$.  \qed
\end{corollary}

These methods have been used in  \cite[\S4]{gross_defcy} to give examples of smoothings of singular Calabi-Yau varieties of dimension $3$ with isolated rational singularities which are not lci.

\subsection{Deformations of a log Calabi-Yau pair}\label{subsection3.3} In this section, unless otherwise specified, we assume that $Y$ is a compact, reduced, normal and Cohen-Macaulay analytic variety of dimension $n\ge 3$ with isolated singularities, and that $D=\bigcup_iD_i\subseteq Y$ is a divisor with simple normal crossings contained in the smooth locus of $Y$.    Let   $\mathbf{Def}_{Y;D_1,\dots, D_k}$ be the functor of log deformations of the pair $(Y,D)$ as defined in \S\ref{subsection2.4}. We begin by listing some standard facts about $\mathbf{Def}_{Y;D_1,\dots, D_k}$, whose proofs are straightforward modifications of the usual proofs in case $Y$ is smooth. (Compare for example \cite[Proposition 3.4.17]{Sernesi}.)

\begin{proposition} Suppose that $(Y,D)$   satisfies the above assumptions. 
\begin{enumerate}
\item[\rm(i)]  The Zariski tangent space to $\mathbf{Def}_{Y;D_1,\dots, D_k}$ is $\mathbb{T}^1_Y(-\log D) = \Ext^1_{\scrO_Y}(\Omega^1_Y(\log D), \scrO_Y)$ and an obstruction space is $\mathbb{T}^2_Y(-\log D) = \Ext^2_{\scrO_Y}(\Omega^1_Y(\log D), \scrO_Y)$.
\item[\rm(ii)]  There are natural isomorphism   $ \mathit{Ext}^i_{\scrO_Y}(\Omega^1_Y(\log D), \scrO_Y) \cong  \mathit{Ext}^i_{\scrO_Y}(\Omega^1_Y, \scrO_Y)$, as $\Omega^1_Y(\log D)$ is locally free around $D$ and is isomorphic to $\Omega^1_Y$ near $Z$. Thus there are natural maps $\mathbb{T}^i_Y(-\log D)\to H^0(Y; T^i_Y) = H^0(Y; \mathit{Ext}^i_{\scrO_Y}(\Omega^1_Y, \scrO_Y))$, $i=1,2$.  
\item[\rm(iii)]  Let $(\mathcal{Y}_k, \mathcal{D}_k)$ be an element of $\mathbf{Def}_{Y;D_1,\dots, D_k}(A_k)$ and define $T^1((\mathcal{Y}_k,\mathcal{D}_k)/A_k)$ in the obvious way. Then 
$$T^1((\mathcal{Y}_k,\mathcal{D}_k)/A_k)\cong \Ext^1_{\scrO_{\mathcal{Y}_k}}( \Omega^1_{\mathcal{Y}_k/\Spec A_k}(\log \mathcal{D}_k), \scrO_{\mathcal{Y}_k}). \qed$$
\end{enumerate}
\end{proposition} 

\begin{definition}\label{defgenT1pairs} (i)    The  \textsl{$T^1$ lifting property holds for the pair $(Y,D)$} if, for every $(\mathcal{Y}_k, \mathcal{D}_k)$ a log   deformation of the pair $(Y,D)$ over $\Spec A_k$ and with  $(\mathcal{Y}_{k-1}, \mathcal{D}_{k-1})$  the  induced  log deformation  over $\Spec A_{k-1}$, the map
$$T^1((\mathcal{Y}_k,\mathcal{D}_k)/A_k) \to  T^1((\mathcal{Y}_{k-1},\mathcal{D}_{k-1})/A_{k-1})$$
is surjective. In this case, by a variant of Theorem~\ref{T1},  $\mathbf{Def}_{Y;D_1,\dots, D_k}$ is  unobstructed.

\smallskip
\noindent (ii)  For a log deformation  $(\mathcal{Y}_k, \mathcal{D}_k)$ of $(Y,D)$ over $\Spec A_k$,  let    $\delta$ be  the analogue of the map defined in Lemma~\ref{lemma3.3}, i.e.\ $\delta$ corresponds to  the connecting homomorphism
$$\begin{CD}
\Ext^1_{\scrO_{\mathcal{Y}_k}}( \Omega^1_{\mathcal{Y}_k/\Spec A_k}(\log \mathcal{D}_k), \scrO_{\mathcal{Y}_{k-1}})  @>{\delta'}>>  \Ext^2_{\scrO_{\mathcal{Y}_k}}( \Omega^1_{\mathcal{Y}_k/\Spec A_k}(\log \mathcal{D}_k), \scrO_Y)\\
@V{\cong}VV @V{\cong}VV \\
\Ext^1_{\scrO_{\mathcal{Y}_{k-1}}}( \Omega^1_{\mathcal{Y}_{k-1}/\Spec A_{k-1}}(\log \mathcal{D}_{k-1}), \scrO_{\mathcal{Y}_{k-1}})  @>{\delta'}>>  \Ext^2_{\scrO_Y}( \Omega^1_Y(\log D), \scrO_Y),
\end{CD}$$
where the vertical isomorphisms are a consequence of the analogue of  Lemma~\ref{Grosslemma2.3} in this setting.
Let $\ell$ be the natural map coming from the Ext spectral sequence:
$$\ell \colon  \mathbb{T}^2_Y(-\log D) \cong \Ext^2_{\scrO_Y}(\Omega^1_Y(\log D), \scrO_Y) \to H^0(Y;\mathit{Ext}^2_{\scrO_Y}(\Omega^1_Y(\log D), \scrO_Y)   \cong H^0(Y;\mathit{Ext}^2_{\scrO_Y}(\Omega^1_Y, \scrO_Y)  ).$$  
 As in Definition~\ref{defgenT1} , the \textsl{generalized $T^1$ lifting property}  holds for the pair $(Y,D)$ if , for every deformation  $(\mathcal{Y}_k, \mathcal{D}_k)$ of $(Y,D)$ over $\Spec A_k$
the restriction of $\ell$ to $\im \delta$ is injective. In this case, the natural variant  of Theorem~\ref{conseqgenT1}  holds.
\end{definition}

\begin{definition} A  \textsl{singular log Calabi-Yau pair}  $(Y,D)$ is a pair satisfying the assumptions of this section, such that  $Y$ is  singular   and   $\omega_Y \cong \scrO_Y(-D)$ (hence in particular $Y$ is Gorenstein).
\end{definition}

Then  we have the following:

\begin{theorem}\label{pairmainthm} Let $(Y,D)$ be a singular log Calabi-Yau pair of dimension $n \ge 3$, such that:
\begin{enumerate}
\item[\rm(i)]  The $\partial\bar\partial$-lemma holds for some resolution of singularities $\hY$ of $Y$ and all of the $k$-fold intersections $D_{i_1}\cap\cdots \cap D_{i_k}$.
\item[\rm(ii)]  The singularities of $Y$ are isolated, lci and Du Bois.
\item[\rm(iii)]  $H^1(Y; \scrO_Y) = 0$.
\end{enumerate} Then the functor $\mathbf{Def}_{Y;D_1,\dots, D_k}$ is  unobstructed.

Without  the assumption in (ii) that the singularities are local complete intersections,   the generalized $T^1$ lifting property holds for the pair $(Y,D)$.  Hence, if  the functor $\mathbf{Def}_{Y;D_1,\dots, D_k}$  is prorepresentable, then the natural analogue of  Theorem~\ref{conseqgenT1} holds for $\mathbf{Def}_{Y;D_1,\dots, D_k}$. 
\end{theorem}
\begin{proof} By assumption, $\omega_Y \cong \scrO_Y(-D)$. By Serre duality, the assumption that $H^1(Y; \scrO_Y) = 0$ implies that $H^{n-1}(Y; \omega_Y) = H^{n-1}(Y; \scrO_Y(-D)) = 0$. Likewise,  $H^{n-i}(Y; \Omega^1_Y(\log D)(-D))$ is Serre dual to $\Ext^i_{\scrO_Y}(\Omega^1_Y(\log D), \scrO_Y)$. If $(\mathcal{Y}, \mathcal{D})$ is a log deformation of $(Y, D)$ over $\Spec A$, where $A$ is a Gorenstein  Artin local $\Cee$-algebra, then the hypothesis that $H^1(Y; \scrO_Y) = 0$ implies that $\scrO_{\mathcal{Y}}(-\mathcal{D}) \cong \omega_{\mathcal{Y}/\Spec A}$. Thus 
$$\Hom_A(H^{n-i}(\mathcal{Y}; \Omega^1_{\mathcal{Y}/\Spec A}(\log\mathcal{D})(-\mathcal{D})), A) \cong \Ext^i_{\scrO_{\mathcal{Y}}}(\Omega^1_{\mathcal{Y}/\Spec A}(\log\mathcal{D}), \scrO_{\mathcal{Y}}).$$ 

There is a natural analogue of Theorem~\ref{proposition3.7} which is   obtained by  replacing  $\Omega^1_{\mathcal{Y}_k/\Spec A_k}$   by  $\Omega^1_{\mathcal{Y}_k/\Spec A_k}(\log\mathcal{D}_k)(-\mathcal{D}_k)$ throughout and making other   obvious modifications,  and the proof goes over verbatim. In the lci case, the map 
$$H^{n-2}(\mathcal{Y}_k; \Omega^1_{\mathcal{Y}_k/\Spec A_k}(\log\mathcal{D})(-\mathcal{D})) \to H^{n-2}(Y; \Omega^1_Y(\log D)(-D))$$ is surjective by the first part of Theorem~\ref{pairsurjthm}. Hence, by the argument used in \S\ref{subsection3.2} to prove Theorem~\ref{mainthm},  $T^1$ lifting holds for the pair $(Y,D)$, so that $\mathbf{Def}_{Y;D_1,\dots, D_k}$ is unobstructed. A similar argument shows that the generalized $T^1$ lifting property holds in the general case.
\end{proof}

\begin{remark}\label{lastremark}  (i) Via Poincar\'e residue, there is an exact sequence
$$0 \to \Omega^1_Y\to \Omega^1_Y(\log D) \to \bigoplus_i\scrO_{D_i} \to 0.$$
Since $\mathit{Ext}^i_{\scrO_Y}(\scrO_{D_i}, \scrO_Y) = 0$ for $i\neq 1$ and $\mathit{Ext}^1_{\scrO_Y}(\scrO_{D_i}, \scrO_Y) \cong N_{D_i/Y}$, the normal bundle of $D_i$ in $Y$, there is an exact sequence
\begin{align*}
0 \to \mathbb{T}^0_Y(-\log D) &\to \mathbb{T}^0_Y \to \bigoplus_iH^0(D_i; N_{D_i/Y}) \to \mathbb{T}^1_Y(-\log D) \to \mathbb{T}^1_Y \to  \bigoplus_iH^1(D_i; N_{D_i/Y})\\
&\to \mathbb{T}^2_Y(-\log D) \to \mathbb{T}^2_Y.
\end{align*}
If $H^1(D_i; N_{D_i/Y}) =0$ for every $i$, i.e.\ if the $D_i$ are stable submanifolds of $Y$, then it is well-known  that the morphism of functors  $\mathbf{Def}_{Y;D_1,\dots, D_k} \to \mathbf{Def}_Y$ is smooth. (Compare \cite[Proposition 3.4.23 and Proposition 2.3.6]{Sernesi}.) In this case, if  $\mathbf{Def}_{Y;D_1,\dots, D_k}$ is unobstructed, then $ \mathbf{Def}_Y$ is  unobstructed as well.

\noindent (ii) Suppose that $Y$ has rational singularities and $\omega_Y^{-1}$ is nef and big. A standard argument shows that $H^{n-1}(Y; \omega_Y) = H^1(Y; \scrO_Y) = 0$ and   $H^1(Y; \scrO_Y(D)) =0$. Thus, in the lci case $\mathbf{Def}_{Y;D_1,\dots, D_k}$ is unobstructed.  If  in addition $D$ is smooth, then $ \mathbf{Def}_Y$ is  unobstructed as well.
\end{remark}

Adapting the method of proof of \cite[Theorem 1.2]{Sano2} gives the following:

\begin{theorem}\label{Fanomainthm} Let $Y$ be a compact   analytic space with isolated lci rational singularities of dimension at least $3$ such  that $\omega_Y^{-1}$ is  nef and big. Then $ \mathbf{Def}_Y$ is  unobstructed. 
\end{theorem} 
\begin{proof}  By the base point free theorem (e.g.\ \cite[(9.3)]{CKM}),    there exists a positive integer  $m$ and a smooth divisor $D$ on $Y$ such that $\omega_Y^{\otimes m} \cong \scrO_Y(-D)$. In particular, $D$ is contained in the smooth locus of $Y$.  Since $\omega_Y^{-1}$ is nef and big, $\pi^*\omega_Y^{-1}$ is also nef and big. By Kawamata-Viehweg vanishing, $H^{n-1}(\hY; \pi^*\omega_Y) =0$. Since $Y$ has rational singularities, $H^{n-1}(Y;\omega_Y) =0$, 
and hence $H^1(Y; \scrO_Y) =0$ as well.  

Next, we claim that $ \mathbf{Def}_{(Y,D)}$ is  unobstructed. By $T^1$ lifting  (Definition~\ref{defgenT1pairs}), it suffices to show that, for a log deformation   $(\mathcal{Y}_k, \mathcal{D}_k)$ of $(Y,D)$ over $\Spec A_k$, the map 
$$  \Ext^1_{\scrO_{\mathcal{Y}_k}}( \Omega^1_{\mathcal{Y}_k/\Spec A_k}(\log \mathcal{D}_k), \scrO_{\mathcal{Y}_k})  \to    \Ext^1_{\scrO_{\mathcal{Y}_{k-1}}}( \Omega^1_{\mathcal{Y}_{k-1}/\Spec A_{k-1}}(\log \mathcal{D}_{k-1}), \scrO_{\mathcal{Y}_{k-1}}) $$
is surjective. By the argument used in \S\ref{subsection3.2} to prove Theorem~\ref{mainthm}, it suffices to prove that the map 
$$H^{n-2}(\mathcal{Y}_k;  \Omega^1_{\mathcal{Y}_k/\Spec A_k}(\log \mathcal{D}_k)\otimes \omega_{\mathcal{Y}_k/\Spec A_k}) \to H^{n-2}(Y; \Omega^1_Y(\log D)\otimes  \omega_Y)$$
is surjective. Note that $\omega_{\mathcal{Y}_k/\Spec A_k}^{\otimes m} \cong \scrO_{\mathcal{Y}_k}(- \mathcal{D}_k)$ since $H^1(Y; \scrO_Y) =0$. 

Let $\nu\colon W \to Y$ be the cyclic cover of degree $m$ branched along $D$ defined by the isomorphism $\omega_Y^{\otimes m} \cong \scrO_Y(-D)$ and let $D'\subseteq W$ be the (reduced) preimage of $D$. Thus   $\nu^*\omega_Y = \scrO_W(-D')$,   $W$ is projective, and $D'$ is an ample divisor on $Z$. Since $Y$ has rational singularities and $\nu$ is \'etale in a neighborhood of $Z$, $W$ has rational singularities as well and $H^{n-1}(W;  \scrO_W(-D')) =0$ by another application of  Kawamata-Viehweg vanishing. Taking the cyclic cover of a deformation  $(\mathcal{Y}_k, \mathcal{D}_k)$ of $(Y,D)$ over $\Spec A_k$ branched along  $\mathcal{D}_k$ yields  a deformation  $(\mathcal{W}_k, \mathcal{D}_k')$ of the pair $(W,D')$ over $\Spec A_k$, with $\nu^* \omega_{\mathcal{Y}_k/\Spec A_k}^{\otimes m}\cong \scrO_{\mathcal{W}_k}(-\mathcal{D}_k')$. By Theorem~\ref{pairsurjthm}, the natural map
\begin{equation*}
H^{n-2}(\mathcal{W}_k;  \Omega^1_{\mathcal{W}_k/\Spec A_k}(\log \mathcal{D}_k')(-\mathcal{D}_k')) \to H^{n-2}(W; \Omega^1_W(\log D')(-D'))\tag{$*$}
\end{equation*}
is surjective. On the other hand, $\Omega^1_W(\log D')(-D') = \nu^*(\Omega^1_Y(\log D)\otimes \omega_Y)$, and hence
\begin{align*}
H^{n-2}(W; \Omega^1_W(\log D')(-D')) &= H^{n-2}(Y; \nu_*\nu^*(\Omega^1_Y(\log D)\otimes \omega_Y)) \\
&= H^{n-2}(Y; \Omega^1_Y(\log D)\otimes \omega_Y\otimes \nu_*\scrO_W).
\end{align*} 
 The direct image $\nu_*\scrO_W \cong \bigoplus_{i=0}^{m-1}\omega_Y^i$ is a direct sum of eigenspaces for the natural $\boldsymbol{\mu}_m$-action. Similarly 
$$H^{n-2}(Y; \Omega^1_Y(\log D)\otimes \omega_Y\otimes \nu_*\scrO_W) \cong \bigoplus_{i=0}^{m-1}H^{n-2}(Y; \Omega^1_Y(\log D)\otimes \omega_Y\otimes\omega_Y^i),$$
and $H^{n-2}(Y; \Omega^1_Y(\log D)\otimes \omega_Y)$ is the trivial $\boldsymbol{\mu}_m$-eigenspace   of $H^{n-2}(Y; \Omega^1_Y(\log D)\otimes \omega_Y\otimes \nu_*\scrO_W)$. Analogous  results hold  for $H^{n-2}(\mathcal{W}_k;  \Omega^1_{\mathcal{W}_k/\Spec A_k}(\log \mathcal{D}_k')(-\mathcal{D}_k'))$. Since the map in $(*)$ is surjective and respects the various eigenspaces,  the map 
$$H^{n-2}(\mathcal{Y}_k;  \Omega^1_{\mathcal{Y}_k/\Spec A_k}(\log \mathcal{D}_k)\otimes \omega_{\mathcal{Y}_k/\Spec A_k}) \to H^{n-2}(Y; \Omega^1_Y(\log D)\otimes  \omega_Y)$$
is surjective as well. Hence  $ \mathbf{Def}_{(Y,D)}$ is  unobstructed.

Finally, since $\omega_Y^{-1}$ is nef and big and $Y$ has rational singularities, the   Kawamata-Viehweg vanishing theorem implies that $H^1(Y; \scrO_Y(D)) =H^1(Y; \omega_Y^{\otimes -m})  =H^1(Y; \omega_Y^{\otimes -m-1}\otimes \omega_Y)=0$ and that $H^2(Y;\scrO_Y) =H^2(Y; \omega_Y^{-1}\otimes \omega_Y) =0$. From the exact sequence
$$0 \to \scrO_Y\to \scrO_Y(D) \to \scrO_D(D) \to 0,$$
it follows that $H^1(D; \scrO_D(D)) = H^1(D; N_{D/Y}) = 0$. Thus $D$ is a stable submanifold of $Y$. Hence, as in Remark~\ref{lastremark}, the morphism $ \mathbf{Def}_{(Y,D)} \to  \mathbf{Def}_Y$ is smooth. Since $ \mathbf{Def}_{(Y,D)}$ is unobstructed, $ \mathbf{Def}_Y$ is  unobstructed as well. 
\end{proof}

\begin{remark}  There is an example of   Sano \cite[Example 2.8]{Sano2} of  a smooth projective variety $Y$ with $\dim Y=3$, such that $\omega_Y^{-1}$ is nef and big but $\mathbb{T}^2_Y = H^2(Y; T_Y) \neq 0$. It is reasonable to expect that such examples exist in the singular case as well.
\end{remark}

\bibliography{zerolimref}
\end{document}